\documentclass[11pt,a4paper]{article}
\usepackage[T1]{fontenc}
\usepackage{lmodern}
\usepackage{amsmath,amssymb,amsthm,mathtools}
\usepackage[a4paper,left=27mm,right=27mm,top=26mm,bottom=27mm]{geometry}
\usepackage{microtype}
\usepackage{enumitem,needspace}
\usepackage[bookmarksnumbered,hyperfootnotes=false,
  pdfborder={0 0 1},linkbordercolor={0 1 0},citebordercolor={0 1 0}]{hyperref}
\newtheorem{maintheorem}{Theorem}

\newtheorem{lemma}{Lemma}[section]
\newtheorem{proposition}[lemma]{Proposition}
\newtheorem{corollary}[lemma]{Corollary}
\theoremstyle{definition}

\theoremstyle{remark}
\newtheorem{remark}[lemma]{Remark}

\makeatletter
\renewenvironment{proof}[1][\proofname]{\par
  \normalfont\topsep6\p@\@plus6\p@\relax
  \trivlist\item[\hskip\labelsep\bfseries #1\@addpunct{.}]\ignorespaces
}{\endtrivlist\@endpefalse}
\makeatother

\DeclareMathOperator{\form}{form}
\DeclareMathOperator{\sform}{sform}
\DeclareMathOperator{\Aut}{Aut}
\DeclareMathOperator{\Inn}{Inn}
\DeclareMathOperator{\Out}{Out}
\DeclareMathOperator{\Comp}{Comp}
\DeclareMathOperator{\SL}{SL}
\DeclareMathOperator{\GL}{GL}

\DeclareMathOperator{\PGL}{PGL}

\newcommand{\F}{\mathfrak F}
\newcommand{\C}{\mathfrak C}
\newcommand{\D}{\mathfrak D}
\newcommand{\X}{\mathfrak X}
\newcommand{\FF}{\mathbf F}
\newcommand{\Sn}{\mathrm S_n}
\newcommand{\Qop}{\mathrm Q}
\newcommand{\Rop}{\mathrm R_0}
\newcommand{\sd}{\leq_{\mathrm{sd}}}
\newcommand{\sn}{\mathrel{\unlhd\!\unlhd}}
\newcommand{\epi}{\twoheadrightarrow}
\newcommand{\iso}{\simeq}
\newcommand{\normal}{\unlhd}
\newcommand{\maps}{\longrightarrow}

\setlist[enumerate]{label=\textup{(\roman*)},leftmargin=2em,itemsep=2pt,topsep=5pt}
\allowdisplaybreaks[1]

\title{On a problem of A.~N.~Skiba and L.~A.~Shemetkov\protect\\
concerning $S_n$-closed subformations\protect\\ of a one-generated formation}
\author{Chenxiao Lin\\[4pt]
\normalsize School of Mathematics and Statistics, Jiangsu Normal University,\\
\normalsize Xuzhou 221116, Jiangsu, China.\\
E-mail: 3020243174@jsnu.edu.cn}
\date{}

\hypersetup{
  pdftitle={On a problem of A. N. Skiba and L. A. Shemetkov concerning Sn-closed subformations of a one-generated formation},
  pdfauthor={Chenxiao Lin},
  pdfsubject={Group Theory (math.GR); MSC 20D10, 20D35, 20F17},
  pdfkeywords={finite group, formation, subnormal subgroup, quasisimple group, central product}
}

\begin{document}
\maketitle

\begin{abstract}
Throughout this paper, all groups are finite. Let $\form(G)$ denote the
formation generated by a group $G$. A formation $\F$ is said to be
$S_n$-closed if every subnormal subgroup of every group in $\F$ belongs
to $\F$. A.~N.~Skiba and L.~A.~Shemetkov posed the problem of proving
that $\form(G)$ has only finitely many $S_n$-closed subformations for
every group $G$ (Problem~14.81 of the \emph{Kourovka Notebook}). In
this paper we propose a negative answer to this problem. We prove that
for the group $G_0=(\SL_2(5)\rtimes C_2)\times\SL_2(5)$ of order
$28800$, where $C_2$ acts on $\SL_2(5)$ by an outer automorphism, the
formation $\form(G_0)$ has an infinite strictly descending chain of
$S_n$-closed subformations. The main tool is a family of central-valued
multilinear maps, defined for every group in $\form(G_0)$ by its outer
actions on its components; their vanishing is preserved under taking
epimorphic images, subdirect products and subnormal subgroups. The
members of the chain are separated by central quotients of the groups
$\SL_2(5)^{2^d-1}\rtimes C_2^d$, $d\geq2$. Using Burichenko's
description of the lattice of subformations of
$\form(\SL_2(5)\rtimes C_2)$, we also prove that every non-empty
$S_n$-closed subformation of $\form(G_0)$ is generated, as an ordinary
formation, by a single group, and that the set of such subformations is
countably infinite. Finally, we determine the intersection of the chain
and give an explicit generator of it.
\end{abstract}

\begingroup
\renewcommand{\thefootnote}{}
\footnotetext{\textit{Keywords:} finite group, formation, subnormal subgroup,
subformation, quasisimple group, central product.\newline
\textit{Mathematics Subject Classification (2020):} 20D10, 20D35, 20F17.}
\endgroup

\section{Introduction}

Throughout this paper, all groups are finite, and all classes of groups
are assumed to be closed under isomorphisms. Recall that a non-empty
class of groups $\F$ is called a \emph{formation} if it is closed under
taking epimorphic images and finite subdirect products. We write
$\form(\X)$ to denote the smallest formation containing a class of
groups $\X$, and we write $\form(G)$ instead of $\form(\{G\})$. Thus a
\emph{one-generated formation} is a formation generated by a single
group, which need not be cyclic.

A formation $\F$ is said to be \emph{$S_n$-closed} provided
\[
 H\sn G\in\F\quad\Longrightarrow\quad H\in\F,
\]
where $H\sn G$ means that $H$ is subnormal in $G$. In other words,
$\F$ is $S_n$-closed if and only if $\Sn\F=\F$, where $\Sn\F$ is the
class of all subnormal subgroups of groups in $\F$ (see
Section~\ref{sec:prelim}). Note that $S_n$-closedness does not mean closure under taking arbitrary
subgroups. All unexplained notation and terminology are standard; the
reader is referred to \cite{BallesterEzquerro}, \cite{DoerkHawkes}
and \cite{ShemetkovSkiba} (see also \cite{SemenchukSkiba,Skiba2011})
if necessary. No
saturation hypothesis is imposed. Allowing the empty formation would
not affect the assertions below.

Bryant, Bryce and Hartley proved that the finite groups in the variety
generated by a finite soluble group form a formation with only finitely
many subformations \cite[Theorem~1.2]{BBH}. Consequently, $\form(G)$
has only finitely many subformations whenever $G$ is soluble. However,
the corresponding question of Gasch\"utz for arbitrary groups $G$,
recorded by A.~N.~Skiba and L.~A.~Shemetkov as Problem~9.59 of the
\emph{Kourovka Notebook} \cite[Problem~9.59, p.~251]{Kourovka}, has a
negative answer: Burichenko constructed a counterexample by using a
double cover of $S_5$, and he determined the corresponding lattice of
subformations \cite{Burichenko2012}.

In \cite[Section~3]{BBH}, Bryant, Bryce and Hartley also described the
six subformations of $\form(\SL_2(5))$. They are generated by $1$,
$C_2$, $A_5$, $A_5\times C_2$, $\SL_2(5)\circ\SL_2(5)$ and $\SL_2(5)$,
respectively, where $\circ$ denotes the central product identifying the
two central involutions. This example shows that central products can
distinguish ordinary formations with the same $S_n$-closure. Arguments with central products and subdirect products
related to those used below can also be found in the papers of
F\"orster and Salomon \cite[Example~4.1 and Lemma~4.2]{ForsterSalomon}
and Vedernikov \cite[Lemma~4 and Corollary~1, pp.~530--532]{Vedernikov}.

Problem~14.81 of the \emph{Kourovka Notebook}, also posed by
A.~N.~Skiba and L.~A.~Shemetkov, is the analogue of Problem~9.59 for
$S_n$-closed subformations: it asks to prove that $\form(G)$ has only
finitely many $S_n$-closed subformations for every group $G$
\cite[Problem~14.81, p.~84]{Kourovka}. By the theorem of Bryant, Bryce
and Hartley quoted above, this is true if $G$ is soluble. Note that the existence of infinitely
many ordinary subformations does not by itself imply that infinitely
many of them are $S_n$-closed. In our construction, this additional
closure property follows from the compatibility of central-valued maps
with restriction to subnormal subgroups. Our main goal here is to prove
the following result.

\begin{maintheorem}\label{thm:main}
There is a finite group $G_0$ of order $28800$ such that
$\F=\form(G_0)$ is $S_n$-closed and contains formations $\C_k$, $k\geq1$,
with
\[
 \C_1\supsetneq\C_2\supsetneq\C_3\supsetneq\cdots,
 \qquad \Sn\C_k=\C_k.
\]
Moreover, $\F$ contains infinitely many pairwise distinct $S_n$-closed
subformations which are themselves generated, as ordinary formations,
by single finite groups.
\end{maintheorem}

In fact, we take
\[
 G_0=(L\rtimes_\theta C_2)\times L,
 \qquad L=\SL_2(\FF_5),
\]
where $\theta$ is the outer involution described in
Lemma~\ref{lem:base}. The key observation is that every group in $\F$
has a perfect derived subgroup whose components are normal and
isomorphic to $L$ or $A_5$. This allows us to define the central-valued
maps $\mu_k^H$ intrinsically, and not in terms of a chosen subdirect
presentation of $H$. The vanishing of these maps is preserved under
taking epimorphic images, subdirect products and subnormal subgroups.

Let $A=L\rtimes_\theta C_2$ and
\begin{equation}\label{eq:sides}
 \mathfrak Q=\form(A),\qquad\mathfrak R=\form(L).
\end{equation}
For formations $\mathfrak U$ and $\mathfrak V$, we write
$\mathfrak U\vee\mathfrak V$ to denote the formation
$\form(\mathfrak U\cup\mathfrak V)$. We also prove the following
properties of the formation $\F$ in Theorem~\ref{thm:main}.

\begin{maintheorem}\label{thm:further}
Every non-empty $S_n$-closed subformation of
$\F=\form(A\times L)$ is generated, as an ordinary formation, by a
single finite group. The set of these subformations is countably
infinite. Moreover, for the chain in Theorem~\ref{thm:main}, if
$\mathfrak E_k=\C_k\cap\mathfrak Q$, then
\[
 \C_k=\sform(\mathfrak E_k)=\mathfrak E_k\vee\mathfrak R
 \qquad(k\geq1).
\]
In particular, every term of this strictly descending chain is
one-generated.
\end{maintheorem}

Now we discuss the relation between our construction and Burichenko's
example. In \cite[pp.~15--16, Theorem~1]{Burichenko2013}, Burichenko
restates \cite[Theorem~2.5]{Burichenko2012}, and this theorem
explicitly covers both stem double covers of $S_5$. By
Lemma~\ref{lem:base}, our group
$A$ satisfies the conditions $|Z(A)|=2$, $A/Z(A)\iso S_5$ and
$Z(A)\leq A'$ stated in that theorem; hence $A$ is among the covers
treated there. Our ambient formation, however, is
$\sform(A)=\form(A\times L)$. In fact, by
Proposition~\ref{prop:ordinary-comparison}, the formation $\form(A)$
itself has only two non-empty $S_n$-closed subformations, and this is
the reason why we pass to its $S_n$-closure.

Note that the proof of Theorem~\ref{thm:main}, including the
strictness of the inclusions after passing to $S_n$-closures, does not
use Burichenko's classification. The same is true for the recovery
formula
\[
 \X=(\X\cap\mathfrak Q)\vee(\X\cap\mathfrak R)
\]
for every $S_n$-closed formation $\X\subseteq\F$ (see
Proposition~\ref{prop:recovery}). On the other hand, the assertion of
Theorem~\ref{thm:further} concerning all $S_n$-closed subformations of
$\F$ is based on Burichenko's classification in the form given in
\cite[Theorem~1]{Burichenko2013}. We do not present the covering group
$A$ or the general formula for $S_n$-closures as new. The point of our
approach is that the central-valued maps are compatible simultaneously
with epimorphic images, subdirect products and subnormal subgroups.
Note also that we do not determine the positions of the formations
$\mathfrak E_k$ in Burichenko's classification.

In Section~\ref{sec:prelim}, we fix notation and the generating group.
In Section~\ref{sec:structure}, we describe the derived subgroups of
the groups in $\F$. The maps $\mu_k^H$ and the formations $\C_k$ are
introduced in Section~\ref{sec:invariants}. In
Section~\ref{sec:proof}, for every $d\geq2$ we construct a group
$H_d\in\C_{d-1}\setminus\C_d$, and we prove Theorem~\ref{thm:main}. In
Section~\ref{sec:comparison}, we compare the ordinary formation
$\form(A)$ with its $S_n$-closure, and in Section~\ref{sec:further} we
prove the recovery formula and Theorem~\ref{thm:further}. Finally, in
Appendix~\ref{app:intersection} we determine the intersection of the
chain and give an explicit generator of it, and in
Appendix~\ref{app:boundary} we show that the hypothesis of
$S_n$-closedness in the recovery formula cannot be omitted.

\section{Preliminaries}\label{sec:prelim}

If $\X$ is a non-empty class of groups, then $\Qop\X$, $\Rop\X$ and
$\Sn\X$ denote, respectively, the classes of all epimorphic images, all
finite subdirect products and all subnormal subgroups of groups in
$\X$. We write $\sform(\X)$ to denote the smallest $S_n$-closed
formation containing $\X$. We write $D\sd X_1\times\cdots\times X_s$
if $D$ is a subgroup of $X_1\times\cdots\times X_s$ whose projection
onto every factor is surjective, that is, $D$ is a subdirect product of
$X_1,\ldots,X_s$. As usual, $Z(G)$ and $G'$ denote the centre and the derived
subgroup of $G$, respectively.

\begin{lemma}\label{lem:closure}
For every non-empty class $\X$,
\begin{equation}\label{eq:closure}
 \form(\X)=\Qop\Rop\X,
 \qquad
 \sform(\X)=\form(\Sn\X)=\Qop\Rop\Sn\X.
\end{equation}
If $G$ is finite and $T$ is the direct product of all subnormal
subgroups of $G$, then $\sform(G)=\form(T)$.
\end{lemma}

\begin{proof}
The first equality is well known (see also \cite[p.~348]{BBH}), and
the second one also follows from \cite[Lemma~8, p.~535]{Vedernikov}.
For the reader's convenience, we give the closure arguments. If
$D\sd\prod_{i=1}^s(X_i/N_i)$, then the full inverse image of $D$ in
$\prod_{i=1}^sX_i$ is subdirect and maps onto $D$. Hence
$\Rop\Qop\X\subseteq\Qop\Rop\X$. Moreover, $\Rop\Rop\X=\Rop\X$, since
iterated subdirect products may be flattened. Therefore
\[
 \Rop\Qop\Rop\X\subseteq\Qop\Rop\Rop\X=\Qop\Rop\X,
 \qquad
 \Qop\Qop\Rop\X=\Qop\Rop\X,
\]
so $\Qop\Rop\X$ is a formation.

Next we show that
\[
 \Sn\Qop\X\subseteq\Qop\Sn\X,
 \qquad
 \Sn\Rop\X\subseteq\Rop\Sn\X.
\]
The first inclusion follows by taking the full inverse image of a
subnormal subgroup under an epimorphism. For the second, project a
subnormal subgroup onto each coordinate: its images are subnormal in
the coordinate groups, and it is subdirect in their product. Since
$\Sn\Sn\X=\Sn\X$, we get
\[
 \begin{aligned}
 \Sn\Qop\Rop\Sn\X
 &\subseteq\Qop\Sn\Rop\Sn\X\\
 &\subseteq\Qop\Rop\Sn\Sn\X
 =\Qop\Rop\Sn\X.
 \end{aligned}
\]
Thus $\Qop\Rop\Sn\X$ is an $S_n$-closed formation containing $\X$, and
it is clearly contained in every $S_n$-closed formation containing
$\X$. This proves the second equality in \eqref{eq:closure}.

Finally, $G$ has only finitely many subnormal subgroups. Each of them
is a quotient of $T$, and $T$, being their direct product, belongs to
the formation they generate. Hence
$\form(T)=\form(\Sn\{G\})=\sform(G)$. The lemma is proved.
\end{proof}

Recall that a group $E$ is said to be \emph{quasisimple} if $E=E'$ and
$E/Z(E)$ is a non-abelian simple group. A \emph{component} of $G$ is a
subnormal quasisimple subgroup of $G$. We write $\Comp(G)$ to denote
the set of all components of $G$. Distinct components commute (see,
for example, \cite[Corollary~3.26(i)]{Craven}). Moreover, if $U$ is a
subnormal subgroup of $G$, then every component of $G$ either is a
component of $U$ or centralizes $U$ (see
\cite[Lemmas~3.24--3.25]{Craven}). We shall also use the following
elementary fact: every proper normal subgroup of a quasisimple group
$E$ is contained in $Z(E)$. Indeed, let $N\normal E$. Then
$NZ(E)/Z(E)$ is normal in the simple group $E/Z(E)$. If
$N\nleq Z(E)$, then $NZ(E)=E$, and so $E=E'=(NZ(E))'=N'\leq N$ since
$E$ is perfect. Thus every proper normal subgroup of $E$ lies in
$Z(E)$.

Let
\begin{equation}\label{eq:L}
 L=\SL_2(\FF_5),\qquad Z=Z(L)=\langle z\rangle,\qquad z=-I.
\end{equation}
It is well known that $|L|=120$, $L=L'$, $|Z|=2$ and $L/Z\iso A_5$.
The perfectness of $L$ can also be checked directly: $L$ is generated
by the upper and lower unitriangular matrices, and conjugation by
$\operatorname{diag}(2,3)$ shows that each of these elementary matrices
is a commutator.

\begin{lemma}\label{lem:base}
The natural homomorphism $\Aut(L)\maps\Aut(L/Z)$ is an isomorphism, and
\[
 \Aut(L)\iso S_5,\qquad \Out(L)\iso C_2.
\]
Conjugation by
\[
 M=\begin{pmatrix}0&1\\2&0\end{pmatrix}\in\GL_2(\FF_5)
\]
induces an outer involution $\theta$ of $L$. Set
$A=L\rtimes_\theta\langle t\rangle$, where $t^2=1$. Then
\begin{equation}\label{eq:A}
 |A|=240,\qquad A'=L,\qquad C_A(L)=Z(A)=Z,
 \qquad A/Z\iso\Aut(L)\iso S_5.
\end{equation}
Every automorphism of $L$ extends to an automorphism of $A$ stabilizing $L$.
\end{lemma}

\begin{proof}
Suppose that an automorphism of $L$ induces the identity on $L/Z$. Then
it has the form $x\mapsto x\varepsilon(x)$ for some homomorphism
$\varepsilon:L\maps Z$. Since $L$ is perfect, $\varepsilon=1$. Hence
the natural map of automorphism groups is injective. A matrix
commuting with both
$\left(\begin{smallmatrix}1&1\\0&1\end{smallmatrix}\right)$ and
$\left(\begin{smallmatrix}1&0\\1&1\end{smallmatrix}\right)$ is scalar.
These matrices belong to $L$, so $C_{\GL_2(5)}(L)$ consists exactly of
the scalar matrices. Therefore conjugation by the elements of
$\GL_2(5)$ induces a subgroup of $\Aut(L)$ of order
$|\PGL_2(5)|=120$. Since $\Aut(A_5)\iso S_5$, it follows that the
natural map is an isomorphism.

Now $M^2=2I$, so $\theta^2=1$. Moreover, $\det M=3$ is a non-square in
$\FF_5$. If conjugation by $M$ coincided on $L$ with conjugation by
some $x\in L$, then $x^{-1}M$ would centralize $L$ and hence would be
scalar; but then $\det M$ would be a square. Thus $\theta$ is an outer
automorphism of order $2$.

The image of the action homomorphism $A\maps\Aut(L)$ contains
$\Inn(L)$ and $\theta$, and so this homomorphism is surjective. Its
kernel has order $2$ and contains $Z$, so it coincides with $Z$. Since
$Z$ is central in $A$, we get the assertions on the centralizer and the
centre. Since $A/L$ is abelian and $L$ is perfect, $A'=L$. Finally,
every automorphism of $L$ is induced by conjugation by some element of
$A$, and conjugation by this element on $A$ is the required extension.
The lemma is proved.
\end{proof}

From now on, we fix
\begin{equation}\label{eq:ambient}
 G_0=A\times L,\qquad |G_0|=28800,\qquad
 \F=\form(G_0)=\form(\{A,L\}).
\end{equation}
The last equality holds since a formation contains $A\times L$ if and
only if it contains both $A$ and $L$.

\begin{lemma}\label{lem:ambient-closed}
The subnormal subgroups of $A$ are $1$, $Z$, $L$ and $A$.
Moreover, $\F=\sform(A)$; in particular, $\F$ is $S_n$-closed.
\end{lemma}

\begin{proof}
Let $U\sn A$. Then the image of $U$ in $A/Z\iso S_5$ is subnormal, so
it is $1$, $A_5$ or $S_5$. If this image is trivial, then $U\leq Z$.
If it is $A_5$, then $UZ=L$ and $L=L'=(UZ)'=U'\leq U$, so $U=L$.
If it is $S_5$, then $UZ=A$ and $L=A'=(UZ)'=U'\leq U$; hence $U=A$,
since $Z\leq L$. Clearly, all four subgroups listed are subnormal
in $A$.

Since $A/L\iso C_2$, the group $Z\iso C_2$ is a quotient of $A$. Hence,
by Lemma~\ref{lem:closure},
\[
 \sform(A)=\form(\{1,Z,L,A\})=\form(\{A,L\})=\F.
\]
The lemma is proved.
\end{proof}

\section{The derived subgroups of the groups in \texorpdfstring{$\F$}{F}}\label{sec:structure}

For a subset $I\subseteq\{1,\ldots,n\}$ and automorphisms
$\alpha_i\in\Aut(L)$, the subgroup
\[
 \Delta_I(L)=\{(\alpha_i(x))_{i\in I}:x\in L\}\leq L^I
\]
is called a \emph{full twisted diagonal subgroup}. In a larger direct
product, we regard it as having identity coordinates outside $I$, and
we call $I$ its \emph{support}.

The following lemma is a special case of the covering-group arguments
in \cite[Lemma~4 and Corollary~1]{Vedernikov}. We give a proof, since
we need an explicit description of the diagonal blocks.

\begin{lemma}\label{lem:subdirect-L}
Let $D\sd L^n$, where $n\geq1$. There is a partition
$\{1,\ldots,n\}=I_1\sqcup\cdots\sqcup I_r$ and full twisted diagonal
subgroups $P_j\leq L^{I_j}$ such that
\[
 D'=P_1\times\cdots\times P_r,\qquad P_j\iso L.
\]
In particular, $D'$ is perfect.
\end{lemma}

\begin{proof}
Let $\overline D$ be the image of $D$ in $L^n/Z^n\iso A_5^n$. A
subdirect subgroup of a direct power of a non-abelian simple group is a
direct product of full twisted diagonal subgroups with disjoint
supports (see \cite[Lemma~2.1]{BryantFoy1995}). For completeness, we
note that this follows by induction from Goursat's lemma. Suppose that
the projection onto the first $n-1$ coordinates has $r$ simple direct
factors. Then the last simple factor is either independent of them, or
it is identified with one of these factors through a simple quotient.
No other case occurs, since a normal subgroup of a direct product of
non-abelian simple groups is a product of some of the direct factors
\cite[Lemma~2.2]{BryantFoy1995}.

By Lemma~\ref{lem:base}, every automorphism of $A_5$ lifts to $L$.
Hence the full inverse image $W$ of $\overline D$ in $L^n$ is a direct
product of block groups $W_j$, where
\[
 W_j=\Delta_{I_j}(L)Z^{I_j}\iso L\times Z^{|I_j|-1}.
\]
For example, after untwisting, the elements of a block of size $s$
have the form $(x,xz_2,\ldots,xz_s)$ with $x\in L$ and $z_i\in Z$.
Thus $W'=P_1\times\cdots\times P_r$ for the corresponding diagonal
subgroups $P_j$. On the other hand, $W=DZ^n$ and $Z^n$ is central.
Therefore $D'=W'$. The lemma is proved.
\end{proof}

\begin{lemma}\label{lem:subdirect-mixed}
Let $D\sd A^a\times L^b$, where $a,b\geq0$, and put $P=D''$.
Then $D'=D''=P$, and $P$ is a direct product of full twisted diagonal
copies of $L$ on pairwise disjoint coordinate supports. Each such copy
is normal in $D$, and no support meets both an $A$-coordinate and an
$L$-coordinate. Furthermore, $D/D'$ is an elementary abelian $2$-group.
\end{lemma}

\begin{proof}
The assertion is trivial if $a+b=0$, so we may suppose that $a+b>0$.
Since $A'=L$, the projection of $D'$ onto every coordinate is $L$.
Hence, by Lemma~\ref{lem:subdirect-L} applied to $D'$,
\[
 P=D''=P_1\times\cdots\times P_r,
\]
where the supports $I_j$ of the diagonal copies $P_j\iso L$ form a
partition of the set of coordinates.

\emph{(1) Every $P_j$ is normal in $D$.}
The subgroup $P$ is characteristic in $D$, and conjugation in the
ambient direct product preserves every coordinate support. Since $P_j$
is the intersection of $P$ with the subgroup of all elements supported
on $I_j$, it follows that $D$ normalizes $P_j$.

\emph{(2) The quotient of the normalizer of a diagonal block by this
block is elementary abelian.}
The coordinate twists can be removed by automorphisms of the coordinate
groups: for $L$-coordinates this is clear, and for $A$-coordinates it
follows from the last assertion of Lemma~\ref{lem:base}. Therefore we
may calculate with an untwisted diagonal copy of $L$ on a support $I$
of size $s$.

If all coordinates in $I$ are of type $A$, then an element normalizes
the diagonal if and only if its coordinates induce the same
automorphism of $L$. Since the kernel of $A\epi\Aut(L)$ is $Z$, this
gives
\begin{equation}\label{eq:normalizer-A}
 N_{A^I}(\Delta_I(L))=\Delta_I(A)Z^I.
\end{equation}
Its quotient by $\Delta_I(L)$ is elementary abelian, since it is
generated by the image of the diagonal involution $t$ and by images of
central involutions.

Now suppose that $I$ contains an $L$-coordinate. Then the common
induced automorphism must be inner. Choose $x\in L$ inducing it. Then
every coordinate differs from $x$ by an element of the centralizer of
$L$ in the coordinate group, and this centralizer is $Z$ for both
types of coordinates. Thus, in this case, the normalizer of the block
is
\begin{equation}\label{eq:normalizer-mixed}
 \Delta_I(L)Z^I,
\end{equation}
and its quotient by $\Delta_I(L)$ is again elementary abelian.
Untwisting back, we obtain the same assertions for the original blocks.

\emph{(3) No support is mixed, and $D'=P$.}
If a support contained coordinates of both types, then
\eqref{eq:normalizer-mixed} would force the projection of $D$ onto an
$A$-coordinate in this support to lie in $L<A$, contrary to the
subdirectness of $D$. Let $G=A^a\times L^b$, and let $G_{I_j}$ be the
subgroup of $G$ consisting of all elements supported on $I_j$.
Coordinatewise conjugation preserves every $G_{I_j}$. Since
$P_j=P\cap G_{I_j}$, every element normalizing $P$ normalizes each
$P_j$; in particular, no permutation of the supports occurs.
Conversely, the product of the block normalizers normalizes $P$.
Therefore
\[
 N_G(P)=\prod_{j=1}^r N_{G_{I_j}}(P_j),
\]
and the quotient of this normalizer by $P$ is elementary abelian. Since
$P\leq D\leq N_G(P)$, it follows that $D/P$ is elementary abelian, and
so $D'\leq P=D''\leq D'$. Hence $D'=P$. The lemma is proved.
\end{proof}

\Needspace{9\baselineskip}
\begin{proposition}\label{prop:structure}
If $H\in\F$, then the following statements hold.
\begin{enumerate}
\item $H/H'$ is an elementary abelian $2$-group.
\item $H'$ is the central product of the components of $H$.
\item Every component of $H$ is normal in $H$ and is isomorphic to $L$
      or $A_5$.
\end{enumerate}
In particular, $H'$ is perfect.
\end{proposition}

\begin{proof}
By Lemma~\ref{lem:closure} and \eqref{eq:ambient}, there is an
epimorphism $f:D\epi H$, where $D\sd A^a\times L^b$ for some
$a,b\geq0$. By Lemma~\ref{lem:subdirect-mixed},
$D'=P_1\times\cdots\times P_r$, where every $P_i\iso L$ is normal in
$D$, and $D/D'$ is elementary abelian. Since $H/H'=f(D)/f(D')$ is an
epimorphic image of $D/D'$, this proves (i).

The kernel of $f|_{P_i}$ is normal in $P_i$, and so it is $1$,
$Z(P_i)$ or $P_i$. Hence every non-trivial image $E_i=f(P_i)$ is
isomorphic to $L$ or $A_5$ and is a normal quasisimple subgroup of
$H$, that is, a component of $H$. These non-trivial images are
pairwise distinct: if two of them coincided, then the common image
would be abelian, since the corresponding factors commute. Moreover,
they commute and generate $f(D')=H'$, so $H'$ is their central
product.

Finally, $H$ has no other components. Indeed, every component $E$ of
$H$ is perfect, and so $E\leq H'$. If $E$ were distinct from all the
$E_i$, then $E$ would commute with every $E_i$, since distinct
components commute, and hence $E$ would commute with $H'$. But then
$E\leq Z(H')$, a contradiction. The proposition is proved.
\end{proof}

\section{Central-valued maps and the formations \texorpdfstring{$\C_k$}{Ck}}\label{sec:invariants}

Let $H\in\F$. For every $E\in\Comp(H)$, we define
\[
 c_E=
 \begin{cases}
 \text{the unique non-identity element of }Z(E),&E\iso L,\\
 1,&E\iso A_5.
 \end{cases}
\]
Since $E\normal H$, the element $c_E$ is fixed by $H$, so
$c_E\in Z(H)$. Conjugation on $E$ induces a homomorphism
\[
 \lambda_E:H\maps\Out(E)\iso (\FF_2,+).
\]
Since there is a unique isomorphism between the two groups of order $2$
in this formula, the homomorphism $\lambda_E$ involves no choice. We
call $\lambda_E$ the \emph{outer-action character} of $E$.

For every positive integer $k$, let
\begin{equation}\label{eq:mu}
 \mu_k^H(h_1,\ldots,h_k)=
 \prod_{E\in\Comp(H)}
 c_E^{\,\lambda_E(h_1)\cdots\lambda_E(h_k)}
 \qquad(h_1,\ldots,h_k\in H).
\end{equation}
Here the exponents lie in $\FF_2$ and are interpreted as $0$ or $1$,
and the empty product is equal to $1$. Let
\[
 V(H)=\langle c_E:E\in\Comp(H)\rangle\leq Z(H).
\]
Then $V(H)$ is an elementary abelian $2$-group. Since every
$\lambda_E$ vanishes on $H'$, the map $\mu_k^H$ factors through a map
\[
 (H/H')^k\maps V(H),
\]
which is symmetric and $k$-linear over $\FF_2$ if its codomain $V(H)$
is written additively. We emphasize that \eqref{eq:mu} is intrinsic to
$H$: it depends neither on a choice of subdirect presentation nor on an
ordering of the components.

\begin{lemma}\label{lem:epimorphism}
Let $f:H\epi J$ be an epimorphism with $H,J\in\F$. For every $k\geq1$,
\begin{equation}\label{eq:naturality}
 f\bigl(\mu_k^H(h_1,\ldots,h_k)\bigr)
 =\mu_k^J\bigl(f(h_1),\ldots,f(h_k)\bigr).
\end{equation}
\end{lemma}

\begin{proof}
Let $E\in\Comp(H)$. Then $f(E)$ is either trivial or a normal
quasisimple subgroup of $J$ isomorphic to $L$ or $A_5$. Distinct
components with non-trivial images have distinct images, since
otherwise their common image would be abelian. Moreover, these images exhaust $\Comp(J)$. Indeed, they
generate $f(H')=J'$, and an additional component, being perfect, would
lie in $J'$ and centralize $J'$, as in the proof of
Proposition~\ref{prop:structure}.

We now compare the contributions to \eqref{eq:naturality} component by
component. If $f(E)\iso L$, then $E\iso L$ and $f|_E$ is an
isomorphism. This isomorphism maps $c_E$ to $c_{f(E)}$ and is
compatible with conjugation, so
\[
 \lambda_{f(E)}(f(h))=\lambda_E(h)\qquad(h\in H).
\]
If $f(E)\iso A_5$, then $f(c_E)=1=c_{f(E)}$, irrespective of the
values of $\lambda_E$. If $f(E)=1$, then the contribution of $E$ is
also killed. Multiplying these contributions, we get
\eqref{eq:naturality}. The lemma is proved.
\end{proof}

\begin{lemma}\label{lem:normal}
Let $U\normal H\in\F$. Then
\begin{equation}\label{eq:components-normal}
 \Comp(U)=\{E\in\Comp(H):E\leq U\},
\end{equation}
and, for every $k\geq1$,
\begin{equation}\label{eq:restriction}
 \mu_k^U(u_1,\ldots,u_k)=\mu_k^H(u_1,\ldots,u_k)
 \qquad(u_1,\ldots,u_k\in U).
\end{equation}
\end{lemma}

\begin{proof}
First, $U\in\F$ by Lemma~\ref{lem:ambient-closed}. A component of $H$
contained in $U$ is a normal quasisimple subgroup of $U$, and so it is
a component of $U$. Conversely, a component of $U$ is subnormal in $H$,
and so it is a component of $H$. This proves
\eqref{eq:components-normal}.

Now let $E\in\Comp(H)$ with $E\nleq U$. Since $E$ and $U$ are normal
in $H$, the subgroup $U\cap E$ is a proper normal subgroup of $E$.
Hence
\[
 [U,E]\leq U\cap E\leq Z(E).
\]
For a fixed $u\in U$, the map $e\mapsto[e,u]$ from $E$ to $Z(E)$ is a
homomorphism, since all its values are central in $E$. Since $E$ is
perfect, this map is trivial. Thus
\begin{equation}\label{eq:omitted}
 E\nleq U\quad\Longrightarrow\quad[U,E]=1
 \quad\Longrightarrow\quad\lambda_E|_U=0.
\end{equation}
For the components $E$ contained in $U$, the elements $c_E$ and the
outer-action characters do not change under restriction. For
the other components, \eqref{eq:omitted} shows that their contribution
to \eqref{eq:mu} is trivial on $U^k$. Hence \eqref{eq:restriction}
holds. The lemma is proved.
\end{proof}

\begin{proposition}\label{prop:formations}
For every $k\geq1$, the class
\begin{equation}\label{eq:Ck}
 \C_k=\{H\in\F:\mu_k^H\equiv1\}
\end{equation}
is an $S_n$-closed formation. Moreover, $\C_{k+1}\subseteq\C_k$.
\end{proposition}

\begin{proof}
Clearly, $\C_k$ contains the identity group, and $\C_k$ is closed under
isomorphisms since \eqref{eq:mu} is intrinsic. By
Lemma~\ref{lem:epimorphism}, $\C_k$ is closed under taking epimorphic
images.

Now let $D\sd H_1\times\cdots\times H_s$, where $H_i\in\C_k$ for all
$i$. Then $D\in\F$, since $\F$ is a formation. For each coordinate
epimorphism $\pi_i:D\epi H_i$, Lemma~\ref{lem:epimorphism} gives
\[
 \pi_i\bigl(\mu_k^D(x_1,\ldots,x_k)\bigr)
 =\mu_k^{H_i}\bigl(\pi_i(x_1),\ldots,\pi_i(x_k)\bigr)=1.
\]
Since the intersection of the kernels of the $\pi_i$ is trivial, we
have $\mu_k^D\equiv1$. Hence $\C_k$ is closed under finite subdirect
products, and so $\C_k$ is a formation.

By Lemma~\ref{lem:normal}, every normal subgroup of a group in $\C_k$
belongs to $\C_k$. Applying this successively along a subnormal chain,
we see that $\C_k$ is $S_n$-closed.

Finally, since $a^2=a$ for all $a\in\FF_2$, we have
\begin{equation}\label{eq:diagonal}
 \mu_{k+1}^H(h_1,\ldots,h_k,h_k)=\mu_k^H(h_1,\ldots,h_k).
\end{equation}
Thus $\mu_{k+1}^H\equiv1$ implies $\mu_k^H\equiv1$, and so
$\C_{k+1}\subseteq\C_k$. The proposition is proved.
\end{proof}

\begin{remark}\label{rem:boolean}
The maps $\mu_k^H$ admit the following elementary interpretation in
terms of polynomial functions. Let $B_H=H/H'$, and write $V(H)$
additively. For $\chi\in B_H^*$, let
\[
 v_\chi=\sum_{\lambda_E=\chi}c_E.
\]
Then the additive form of \eqref{eq:mu} is
\[
 \widetilde\mu_k^H(b_1,\ldots,b_k)
 =\sum_{\chi\ne0}v_\chi\chi(b_1)\cdots\chi(b_k).
\]
Equivalently, $\mu_k^H\equiv1$ if and only if
$\sum_E c_E f(\lambda_E)=0$ for every polynomial function
$f:B_H^*\maps\FF_2$ of degree at most $k$ with $f(0)=0$.
Indeed, products of evaluation functions span this space, and a product
of positive degree less than $k$ can be padded by repeated factors,
since $x^2=x$ on $\FF_2$. The condition $f(0)=0$ cannot be omitted.
Of course, this is only a reformulation of the definition of $\C_k$;
the closure properties of $\C_k$ rest on Lemmas~\ref{lem:epimorphism}
and~\ref{lem:normal}.
\end{remark}

\begin{remark}\label{rem:closure-scope}
Every soluble group in $\F$ is an elementary abelian $2$-group.
Indeed, by Proposition~\ref{prop:structure}, the derived subgroup of
such a group is both soluble and perfect, and so it is trivial. Thus
$C_3,C_4\notin\F$, whereas $L,C_2\in\C_k$ for every $k$. In
particular, neither $\F$ nor any $\C_k$ is closed under taking
subgroups, since $C_3\leq L$. These formations are not saturated
either, since $C_4/\Phi(C_4)\iso C_2$, where $\Phi(C_4)$ denotes the
Frattini subgroup of $C_4$.
\end{remark}

\section{The groups \texorpdfstring{$H_d$}{Hd} and the proof of Theorem A}\label{sec:proof}

For an integer $d\geq2$, let
\[
 B_d=(\FF_2^d,+),\qquad
 \Lambda_d=B_d^*\setminus\{0\},\qquad r_d=2^d-1.
\]
Here $B_d^*$ is the vector space of all linear maps $B_d\maps\FF_2$.
For every $\lambda\in\Lambda_d$, let $L_\lambda$ be a copy of $L$.
Let $b\in B_d$ act on $\prod_{\lambda\in\Lambda_d}L_\lambda$ by
\[
 (x_\lambda)_{\lambda\in\Lambda_d}
 \longmapsto
 \bigl(\theta^{\lambda(b)}(x_\lambda)\bigr)_{\lambda\in\Lambda_d},
\]
and let
\begin{equation}\label{eq:Pd}
 P_d=L^{\Lambda_d}\rtimes B_d.
\end{equation}
For every $\lambda\in\Lambda_d$, the homomorphism
\[
 \rho_\lambda:P_d\maps A,\qquad
 ((x_\nu)_\nu,b)\longmapsto (x_\lambda,t^{\lambda(b)})
\]
is surjective. The intersection of the kernels of all $\rho_\lambda$ is
trivial: every $L$-coordinate of an element of this intersection is
trivial, and the non-zero linear forms separate the elements of $B_d$.
Hence
\begin{equation}\label{eq:Pd-subdirect}
 P_d\sd A^{r_d},\qquad P_d\in\form(A).
\end{equation}

The elementary abelian group $Z^{\Lambda_d}$ is central in $P_d$,
since every coordinate action fixes $z$. Let
\begin{equation}\label{eq:Kd}
 K_d=\left\{(z^{a_\lambda})_{\lambda\in\Lambda_d}:
 a_\lambda\in\FF_2,\quad
 \sum_{\lambda\in\Lambda_d}a_\lambda=0\right\},
 \qquad H_d=P_d/K_d.
\end{equation}
It follows from \eqref{eq:Pd-subdirect} that
\begin{equation}\label{eq:Hd-member}
 H_d\in\form(A)\subseteq\F.
\end{equation}

\begin{lemma}\label{lem:witness}
For every integer $d\geq2$,
\[
 H_d\in\C_{d-1}\setminus\C_d.
\]
More precisely, $\mu_k^{H_d}\equiv1$ for $1\leq k<d$, whereas
$\mu_d^{H_d}\not\equiv1$.
\end{lemma}

\begin{proof}
Since $B_d$ is abelian and every copy of $L$ is perfect,
$P_d'=L^{\Lambda_d}$. The intersection of $K_d$ with each coordinate
copy $L_\lambda$ is trivial, since a non-identity element of
$Z^{\Lambda_d}$ supported on one coordinate has odd weight. Hence every
$L_\lambda$ embeds in $H_d$ as a normal quasisimple subgroup. As in
the proof of Proposition~\ref{prop:structure}, the images of the
$L_\lambda$ are all the components of $H_d$, since they commute and
generate $H_d'$.

Let $z_*$ denote the image in $H_d$ of the central involution of any
one coordinate. All such images coincide, since the product of two
coordinate involutions lies in $K_d$; and $z_*\neq1$, since a single
coordinate involution does not lie in $K_d$. Since
$K_d\leq P_d'=L^{\Lambda_d}$, we have
\begin{equation}\label{eq:Hd-abelianization}
 H_d/H_d'\iso P_d/P_d'\iso B_d.
\end{equation}
The elements of the base group act on each component by inner
automorphisms, whereas $b\in B_d$ induces $\theta^{\lambda(b)}$ on the
component arising from $L_\lambda$. Thus, under the identification
\eqref{eq:Hd-abelianization}, the outer-action character of this
component is precisely $\lambda$.

Let $h_1,\ldots,h_k\in H_d$, and let $b_1,\ldots,b_k$ be their images
in $B_d$. Then \eqref{eq:mu} becomes
\begin{equation}\label{eq:parity}
 \mu_k^{H_d}(h_1,\ldots,h_k)
 =z_*^{\,N_k(b_1,\ldots,b_k)},\qquad
 N_k(b_1,\ldots,b_k)=
 \sum_{\lambda\in\Lambda_d}
 \lambda(b_1)\cdots\lambda(b_k)\in\FF_2.
\end{equation}
The exponent is the parity of the number of solutions of the system
\begin{equation}\label{eq:system}
 \lambda(b_1)=\cdots=\lambda(b_k)=1
 \qquad\text{in }\lambda\in B_d^*.
\end{equation}
The zero functional is not a solution, so replacing $B_d^*$ by
$\Lambda_d$ does not change this number. If \eqref{eq:system} is
consistent and the vectors $b_i$ span a subspace of dimension $s$,
then it has exactly $2^{d-s}$ solutions; otherwise it has none. If
$k<d$, then $s\leq k<d$, so this number is even. Hence
$\mu_k^{H_d}\equiv1$ for all $1\leq k<d$.

Now let $k=d$, and let $b_1,\ldots,b_d$ be a basis of $B_d$. Then
there is exactly one linear functional which takes the value $1$ on
every basis vector. Thus $N_d(b_1,\ldots,b_d)=1$, and for any
preimages $h_i$ of the $b_i$ we have
\[
 \mu_d^{H_d}(h_1,\ldots,h_d)=z_*\neq1.
\]
The lemma is proved.
\end{proof}

\begin{proof}[Proof of Theorem~\ref{thm:main}]
By \eqref{eq:ambient}, the group $G_0$ has order $28800$, and
$\F=\form(G_0)$ is $S_n$-closed by Lemma~\ref{lem:ambient-closed}.
By Proposition~\ref{prop:formations}, the formations $\C_k$ form a
descending chain of $S_n$-closed subformations of $\F$. By
Lemma~\ref{lem:witness}, $H_{k+1}\in\C_k\setminus\C_{k+1}$ for every
$k\geq1$. Hence all the inclusions are strict.

Now we prove the last assertion of the theorem. Let
$\D_d=\sform(H_d)$ for $d\geq2$. Since $H_d\in\C_{d-1}\subseteq\F$ and
both $\C_{d-1}$ and $\F$ are $S_n$-closed formations, we have
\[
 \D_d\subseteq\C_{d-1}\subseteq\F.
\]
If $j>i\geq2$, then
\[
 \D_j\subseteq\C_{j-1}\subseteq\C_i,
 \qquad H_i\in\D_i\setminus\C_i.
\]
Thus $\D_i\neq\D_j$. Finally, let $T_d$ be the direct product of all
subnormal subgroups of $H_d$. Then $T_d$ is a finite group, and
$\D_d=\form(T_d)$ by Lemma~\ref{lem:closure}. Therefore every $\D_d$
is one-generated even as an ordinary formation. The theorem is proved.
\end{proof}

\begin{remark}\label{rem:orders}
The groups distinguishing the successive inclusions have orders
\[
 |H_d|=\frac{120^{r_d}\,2^d}{2^{r_d-1}}
       =2^{d+1}60^{\,2^d-1}.
\]
In particular, $|H_2|=1\,728\,000$. Thus their orders grow with $d$,
whereas the generating group $G_0$ is fixed. We do not consider here
the question of the smallest possible order of a group generating a
counterexample.
\end{remark}

\begin{remark}\label{rem:first-example}
Let $d=2$. Then the non-zero linear forms are $x$, $y$ and $x+y$.
Their first-degree sum is zero. On the ordered basis $(e_1,e_2)$, only
$x+y$ contributes to the second-degree sum, which is therefore equal
to one. Thus $H_2\in\C_1\setminus\C_2$. In higher dimensions, the same
phenomenon is expressed by the number of solutions of
\eqref{eq:system}, so that no computation in $H_d$ itself is needed.
\end{remark}

\begin{remark}\label{rem:determinant}
With respect to the standard coordinates on $B_d=\FF_2^d$, the counting
argument in the proof of Lemma~\ref{lem:witness} also gives
\[
 N_d(b_1,\ldots,b_d)=\det[b_1\ \cdots\ b_d]\in\FF_2.
\]
Indeed, if the vectors are linearly dependent, then their span has
dimension less than $d$, so the number of solutions of
\eqref{eq:system} is either zero or even, and both sides are zero. If
they form a basis, then both sides are equal to one. Consequently,
$N_d$ is the non-zero alternating $d$-linear form on $B_d$, whereas
$N_k=0$ for $k<d$ by Lemma~\ref{lem:witness}.
\end{remark}

\section{The formation \texorpdfstring{$\form(A)$}{form(A)} and its
\texorpdfstring{$S_n$}{Sn}-closure}\label{sec:comparison}

Recall that $\mathfrak Q=\form(A)$. By Burichenko's theorem,
$\mathfrak Q$ has infinitely many ordinary subformations (see
\cite{Burichenko2012} and its restatement in
\cite[Theorem~1]{Burichenko2013}). The difference between this fact
and Theorem~\ref{thm:main} can be seen without using the
classification of that lattice.

\begin{proposition}\label{prop:ordinary-comparison}
One has $A_5\notin\mathfrak Q$. The only non-empty $S_n$-closed
subformations of $\mathfrak Q$ are $\form(1)$ and $\form(C_2)$.
Consequently,
\[
 \mathfrak Q\subsetneq\sform(A)=\F.
\]
\end{proposition}

\begin{proof}
First suppose that $A_5\in\mathfrak Q$. Then, by
Lemma~\ref{lem:closure}, there is an epimorphism $f:D\epi A_5$, where
$D\sd A^n$ for some $n\geq1$. By Lemma~\ref{lem:subdirect-mixed},
$D'=P_1\times\cdots\times P_r$, where every $P_j\iso L$ is a full
twisted diagonal subgroup which is normal in $D$. Every $f(P_j)$ is
normal in $A_5$, and so it is either trivial or equal to $A_5$. Since
$f(D')=A_5$ and these images commute, exactly one of them, say
$f(P_j)$, is non-trivial. The kernel of $f|_{P_j}$ is $Z(P_j)$, so
this restriction induces an isomorphism
\[
 \overline f:P_j/Z(P_j)\maps A_5.
\]

Let $i$ be a coordinate in the support of $P_j$. Then the coordinate
map $\pi_i|_{P_j}:P_j\maps L$ is an isomorphism. Since $D$ is
subdirect, there is an element $x\in D$ with $\pi_i(x)=t$, where $t$ is
the involution in Lemma~\ref{lem:base}. Under this coordinate
isomorphism, conjugation by $x$ on $P_j$ corresponds to conjugation by
$t$ on $L$, that is, to $\theta$. Hence the automorphism of
$P_j/Z(P_j)$ induced by $x$ is outer, since the isomorphism
$\Aut(L)\maps\Aut(L/Z)$ in Lemma~\ref{lem:base} maps $\Inn(L)$ onto
$\Inn(L/Z)$. On the other hand,
\[
 \overline f\bigl(xpx^{-1}Z(P_j)\bigr)
 =f(x)\,\overline f\bigl(pZ(P_j)\bigr)\,f(x)^{-1}
 \qquad(p\in P_j).
\]
Thus $\overline f$ transforms this automorphism into an inner
automorphism of $A_5$, a contradiction. Hence $A_5\notin\mathfrak Q$,
and so also $L\notin\mathfrak Q$, since $L/Z\iso A_5$.

Now let $\mathfrak B\subseteq\mathfrak Q$ be a non-empty $S_n$-closed
formation. Suppose that some group $H\in\mathfrak B$ is non-abelian.
Then, by Proposition~\ref{prop:structure}, $H$ has a component
$E\normal H$ isomorphic to $L$ or $A_5$. Since $\mathfrak B$ is closed
under taking subnormal subgroups and quotients, we get
$A_5\in\mathfrak B\subseteq\mathfrak Q$, contrary to what we have just
proved. Therefore every group in $\mathfrak B$ is abelian, and so,
again by Proposition~\ref{prop:structure}, elementary abelian of
exponent dividing $2$. Hence $\mathfrak B\subseteq\form(C_2)$. If
$\mathfrak B$ contains a non-identity group, then it contains a quotient
isomorphic to $C_2$, so $\mathfrak B=\form(C_2)$. Otherwise,
$\mathfrak B=\form(1)$. Both possibilities occur, since $C_2$ is a
quotient of $A$. Finally, $L\in\F\setminus\mathfrak Q$, and so
Lemma~\ref{lem:ambient-closed} gives the asserted proper inclusion.
The proposition is proved.
\end{proof}

\begin{remark}\label{rem:closure-comparison}
The ordinary subformations
\[
 \mathfrak E_k=\C_k\cap\mathfrak Q\qquad(k\geq1)
\]
form a strictly descending chain. Indeed, by the construction in
Section~\ref{sec:proof}, $H_{k+1}\in\mathfrak Q$, and so
Lemma~\ref{lem:witness} gives
$H_{k+1}\in\mathfrak E_k\setminus\mathfrak E_{k+1}$. By
Proposition~\ref{prop:ordinary-comparison}, no $\mathfrak E_k$ is
$S_n$-closed, since $\mathfrak E_k$ contains the non-abelian group
$H_{k+1}$. Nevertheless, their $S_n$-closures remain distinct. More
precisely, let $\widehat{\mathfrak E}_k=\sform(\mathfrak E_k)$. By
monotonicity, and since $\C_k$ is $S_n$-closed, we have
\[
 \widehat{\mathfrak E}_{k+1}\subseteq\widehat{\mathfrak E}_k
 \subseteq\C_k\subseteq\F,
 \qquad
 H_{k+1}\in\widehat{\mathfrak E}_k
 \setminus\widehat{\mathfrak E}_{k+1}.
\]
For the last assertion, we use the inclusion
$\widehat{\mathfrak E}_{k+1}\subseteq\C_{k+1}$. Thus the chain remains
strictly descending after passing to $S_n$-closures in the fixed
formation $\sform(A)$. Note that this argument does not use
Burichenko's classification. In fact,
$\widehat{\mathfrak E}_k=\C_k$ by Corollary~\ref{cor:chain-recovery}
below.
\end{remark}

\section{\texorpdfstring{$S_n$}{Sn}-closed subformations of
\texorpdfstring{$\F$}{F} and the proof of Theorem B}\label{sec:further}

We retain the notation $\mathfrak Q=\form(A)$ and
$\mathfrak R=\form(L)$, and for formations $\mathfrak U$ and
$\mathfrak V$ we write
$\mathfrak U\vee\mathfrak V=\form(\mathfrak U\cup\mathfrak V)$. Note
that all joins in this section are joins of ordinary formations. First
we prove the recovery formula; this part does not use any
classification of subformations.

\begin{lemma}\label{lem:mixed-splitting}
If $D\sd A^a\times L^b$, where $a,b\geq0$, then
\[
 D\iso L^s\times D_0
\]
for some $s\geq0$ and some $D_0\in\mathfrak Q$.
\end{lemma}

\begin{proof}
By Lemma~\ref{lem:subdirect-mixed}, $D'$ is a direct product of full
twisted diagonal copies of $L$, and every support consists either of
$A$-coordinates only or of $L$-coordinates only. In view of
Lemma~\ref{lem:base}, we may remove all coordinate twists. A diagonal
block whose support $I$ consists of $L$-coordinates will be called a
\emph{pure $L$-block}. For every such support $I$, we fix a reference
coordinate $i_I\in I$. Then \eqref{eq:normalizer-mixed} gives
\[
 N_{L^I}(\Delta_I(L))
 =\Delta_I(L)\times Z^{I\setminus\{i_I\}},
\]
where the elements of the second factor have reference coordinate $1$.

Let $P_L$ be the product of all pure $L$-blocks, and let
\[
 D_0=\{x\in D:\pi_{i_I}(x)=1\text{ for every pure }L\text{-block }I\}.
\]
Then $D_0$ is the intersection of the kernels of some coordinate
homomorphisms. Moreover, $P_L\leq D$ and $P_L\iso L^s$, where $s$ is
the number of pure $L$-blocks. For every $x\in D$, there is a unique
element $p\in P_L$ whose reference coordinates coincide with those of
$x$, and then $p^{-1}x\in D_0$. Hence $D=P_LD_0$ and $P_L\cap D_0=1$.
By the displayed formula for the normalizer, all coordinates of an
element of $D_0$ in a pure $L$-block lie in $Z$. Since $P_L$ is trivial
on all other coordinates, $[P_L,D_0]=1$. Thus $D=P_L\times D_0$.

Multiplication by an element of $P_L$ does not change the
$A$-coordinates. Consequently, $D_0$ still projects onto every
$A$-coordinate, and its projections onto all other coordinates are $1$
or $Z\iso C_2$. Omitting the trivial coordinates, we see that $D_0$ is
a subdirect product of copies of $A$ and $C_2$. Since $C_2$ is a
quotient of $A$, it follows that $D_0\in\mathfrak Q$. Clearly, the
argument also covers the case where there are no pure $L$-blocks. The
lemma is proved.
\end{proof}

\begin{proposition}\label{prop:recovery}
Every non-empty $S_n$-closed subformation $\X\subseteq\F$ satisfies
\begin{equation}\label{eq:recovery}
 \X=(\X\cap\mathfrak Q)\vee(\X\cap\mathfrak R).
\end{equation}
Consequently, the map
\[
 \X\longmapsto(\X\cap\mathfrak Q,\X\cap\mathfrak R)
\]
is an order embedding into the product of the ordinary subformation
lattices of $\mathfrak Q$ and $\mathfrak R$.
\end{proposition}

\begin{proof}
Clearly, the right-hand side of \eqref{eq:recovery} is contained in
$\X$. Conversely, let $H\in\X$. By Lemma~\ref{lem:closure}, there is
an epimorphism $f:D\epi H$, where $D\sd A^a\times L^b$. By
Lemma~\ref{lem:mixed-splitting}, $D=P_L\times D_0$, where
$P_L\iso L^s$ and $D_0\in\mathfrak Q$. The subgroups
\[
 R_H=f(P_L),\qquad U_H=f(D_0)
\]
are normal in $H$, they commute, and $H=R_HU_H$. Moreover,
$R_H\in\mathfrak R$ and $U_H\in\mathfrak Q$. Since $\X$ is closed
under taking normal subgroups, both $R_H$ and $U_H$ belong to $\X$.
Therefore $H$, being an epimorphic image of $R_H\times U_H$ under the
multiplication map, belongs to the right-hand side of
\eqref{eq:recovery}. Note that $R_H\cap U_H$ need not be trivial.

Finally, the map $\X\mapsto(\X\cap\mathfrak Q,\X\cap\mathfrak R)$ is
clearly order-preserving. Conversely, if $\X$ and $\mathfrak Y$ are
non-empty $S_n$-closed subformations of $\F$ such that
$\X\cap\mathfrak Q\subseteq\mathfrak Y\cap\mathfrak Q$ and
$\X\cap\mathfrak R\subseteq\mathfrak Y\cap\mathfrak R$, then
$\X\subseteq\mathfrak Y$ by \eqref{eq:recovery}. The proposition is
proved.
\end{proof}

\begin{corollary}\label{cor:chain-recovery}
For every $k\geq1$, put $\mathfrak E_k=\C_k\cap\mathfrak Q$. Then
\begin{equation}\label{eq:chain-recovery}
 \C_k=\mathfrak E_k\vee\mathfrak R=\sform(\mathfrak E_k).
\end{equation}
\end{corollary}

\begin{proof}
The only component of $L$ is $L$ itself, and its outer-action
character is zero. Hence $L\in\C_k$, and so $\mathfrak R\subseteq\C_k$. Therefore the first equality follows from
Proposition~\ref{prop:recovery}. Next, the group $H_{k+1}$ belongs to
$\mathfrak E_k$ and has a normal component isomorphic to $L$. Hence
$\mathfrak R\subseteq\sform(\mathfrak E_k)$, and so
$\mathfrak E_k\vee\mathfrak R\subseteq\sform(\mathfrak E_k)$. The
reverse inclusion holds since $\C_k$ is $S_n$-closed and contains
$\mathfrak E_k$. The corollary is proved.
\end{proof}

In the proof of the first assertion of Theorem~\ref{thm:further}, we
use the known description of the lattice of all subformations of
$\mathfrak Q$. For $\mathfrak R$, we need only its $S_n$-closed
subformations, which are described directly in the following lemma.
This description agrees with the six-element lattice of all
subformations of $\mathfrak R$ given in \cite[Section~3]{BBH}.

\begin{lemma}\label{lem:pure-L-closed}
The formation $\mathfrak R$ is $S_n$-closed. Its non-empty $S_n$-closed
subformations are exactly
\begin{equation}\label{eq:pure-L-list}
 \form(1),\quad\form(C_2),\quad\form(A_5),\quad
 \form(A_5\times C_2),\quad\mathfrak R.
\end{equation}
Moreover, if $H\in\mathfrak R$, then every $\lambda_E$, for
$E\in\Comp(H)$, is zero.
\end{lemma}

\begin{proof}
The subgroup $\Delta(L)Z^2$ is subdirect in $L^2$, and its quotient
by $\Delta(L)$ is isomorphic to $C_2$. Hence $C_2\in\mathfrak R$.
The subnormal subgroups of $L$ are $1$, $Z$ and $L$, as one sees by
passing to $L/Z\iso A_5$ and using the perfectness of $L$. Thus
$\sform(L)=\mathfrak R$ by Lemma~\ref{lem:closure}.

Let $D\sd L^b$. Then, by the proof of Lemma~\ref{lem:mixed-splitting}
with $a=0$, $D=P_L\times C$, where $P_L\iso L^s$ and $C$ is a central
elementary abelian $2$-group. By Lemma~\ref{lem:closure}, every
$H\in\mathfrak R$ is an epimorphic image of such a group $D$, and so
$H=H'W$, where $W$ is a central elementary abelian
$2$-group and $H'$ is a central product of the images of the factors
of $P_L$. Therefore conjugation induces inner automorphisms on every
component of $H$, which proves the last assertion.

Now let $\mathfrak Y\subseteq\mathfrak R$ be a non-empty $S_n$-closed
formation. If some group in $\mathfrak Y$ has a component isomorphic to
$L$, then $L\in\mathfrak Y$, and so $\mathfrak Y=\mathfrak R$.
Otherwise, for every $H\in\mathfrak Y$, all components of $H$ are
isomorphic to $A_5$. Since these components have trivial centres,
$H'\iso A_5^r$ and $H'\cap W=1$. Hence
\[
 H\iso A_5^r\times C_2^s.
\]
Depending on which of these two types of direct factors occur, and
using projections and direct products, we see that $\mathfrak Y$ is
one of the first four formations in \eqref{eq:pure-L-list}.

It remains to note that the class of all groups $A_5^r\times C_2^s$
coincides with $\form(A_5\times C_2)$. To check closure under
subdirect products, project a subdirect subgroup onto its
$A_5$-coordinates. The image is a product of diagonal simple factors
and is perfect. The derived subgroup has the same image and is trivial
on the $C_2$-coordinates; the remaining factor is a central elementary
abelian $2$-group. Closure under quotients follows from the structure
of normal subgroups of a direct product of non-abelian simple groups
and a central elementary abelian group. The same arguments apply when
one type of factor is absent. The normal subgroups of
$A_5^r\times C_2^s$ are products of some of its simple direct factors
and a subgroup of its central factor. Iterating along a subnormal
chain, we see that each of the first four formations in
\eqref{eq:pure-L-list} is $S_n$-closed. The lemma is proved.
\end{proof}

\begin{lemma}\label{lem:finite-cover}
Let $\mathfrak U$ be a formation. Suppose that there are finitely many
proper subformations $\mathfrak M_1,\ldots,\mathfrak M_r$, with $r\geq1$, such that
every proper subformation of $\mathfrak U$ is contained in one of
them. Then $\mathfrak U$ is generated by a single finite group.
The trivial formation has the same property.
\end{lemma}

\begin{proof}
For every $i$, choose $G_i\in\mathfrak U\setminus\mathfrak M_i$, and
let $G=\prod_{i=1}^rG_i$. If $\form(G)$ were a proper subformation of
$\mathfrak U$, then it would be contained in some $\mathfrak M_i$.
But $G_i$ is a quotient of $G$, a contradiction. Thus
$\mathfrak U=\form(G)$. Finally, the trivial formation is $\form(1)$.
The lemma is proved.
\end{proof}

Now we state the part of Burichenko's classification which is used
below. In the notation of \cite[Theorem~1, p.~16]{Burichenko2013}, the
distinct non-empty subformations of $\mathfrak Q$ are
\[
 \form(1),\quad\form(C_2),\quad\mathfrak P=\form(S_5),\quad
 \mathfrak K_*,\quad\mathfrak K_i,\quad\mathfrak H_i\quad(i\geq0).
\]
Here $\mathfrak K_*$ denotes the formation which is denoted by
$\overline{\mathfrak K}$ in \cite{Burichenko2013}. All strict
inclusions between these formations are given, up to transitivity, by
\begin{equation}\label{eq:known-lattice}
 \begin{gathered}
 \form(1)\subsetneq\form(C_2)\subsetneq\mathfrak P
     \subsetneq\mathfrak K_*,\qquad \mathfrak K_0=\mathfrak Q,\\
 \mathfrak K_*\subsetneq\mathfrak K_i,\qquad
 \mathfrak K_*\subsetneq\mathfrak H_i\quad(i\geq0),\\
 \mathfrak K_j\subsetneq\mathfrak K_i,\qquad
 \mathfrak H_j\subsetneq\mathfrak H_i\quad(i<j),\qquad
 \mathfrak H_j\subsetneq\mathfrak K_i\quad(i\leq j).
 \end{gathered}
\end{equation}
In particular, no $\mathfrak K_j$ is contained in an $\mathfrak H_i$.
Note that \eqref{eq:known-lattice} is taken from
\cite{Burichenko2013}; it is not deduced from the results of
Section~\ref{sec:invariants}.

\begin{corollary}\label{cor:Q-one-generated}
Every non-empty subformation of $\mathfrak Q$ is generated by a
single finite group.
\end{corollary}

\begin{proof}
Every proper subformation of $\mathfrak H_i$ is contained in
$\mathfrak H_{i+1}$, and every proper subformation of $\mathfrak K_i$
is contained in $\mathfrak K_{i+1}$ or in $\mathfrak H_i$. Moreover,
each non-trivial member of the finite chain at the bottom of
\eqref{eq:known-lattice} has a greatest proper subformation. Hence
Lemma~\ref{lem:finite-cover} applies in each case. The corollary is
proved.
\end{proof}

\begin{proof}[Proof of Theorem~\ref{thm:further}]
Let $\X$ be a non-empty $S_n$-closed subformation of $\F$. By
Corollary~\ref{cor:Q-one-generated},
$\X\cap\mathfrak Q=\form(U)$ for some group $U\in\mathfrak Q$. Since
both $\X$ and $\mathfrak R$ are $S_n$-closed, so is
$\X\cap\mathfrak R$; hence, by Lemma~\ref{lem:pure-L-closed},
$\X\cap\mathfrak R=\form(V)$, where
\begin{equation}\label{eq:five-generators}
 V\in\{1,C_2,A_5,A_5\times C_2,L\}.
\end{equation}
Now Proposition~\ref{prop:recovery} gives
\begin{equation}\label{eq:one-generator-all}
 \X=\form(U)\vee\form(V)=\form(U\times V).
\end{equation}
Since there are only countably many isomorphism types of finite
groups, there are at most countably many such formations. On the
other hand, the strictly descending chain in Theorem~\ref{thm:main}
shows that there are infinitely many of them. Hence their number is
countably infinite.

The equalities for $\C_k$ are proved in
Corollary~\ref{cor:chain-recovery}. More explicitly, if
$\mathfrak E_k=\form(U_k)$, then
\begin{equation}\label{eq:chain-generator}
 \C_k=\form(U_k\times L).
\end{equation}
Note that the groups $U$ and $U_k$ are obtained here by existence
arguments only; in particular, we do not assert that one can take
$U_k=H_{k+1}$. The theorem is proved.
\end{proof}

\appendix
\section{The intersection of the descending chain}\label{app:intersection}

In this appendix we use the description \eqref{eq:known-lattice} of
the lattice of subformations of $\mathfrak Q$. We retain its notation
$\mathfrak P$ and $\mathfrak K_*$, and let
\[
 \C_\infty=\bigcap_{k\geq1}\C_k.
\]

\begin{lemma}\label{lem:S5-derived}
If $H\in\form(S_5)$, then $H'$ is a direct product of copies of
$A_5$. In particular, $Z(H')=1$.
\end{lemma}

\begin{proof}
By Lemma~\ref{lem:closure}, $H$ is an epimorphic image of some
$D\sd S_5^n$. Then $D'\sd A_5^n$. Since a subdirect product of copies
of a non-abelian simple group is a direct product of full twisted
diagonal copies of this group \cite[Lemma~2.1]{BryantFoy1995}, we have
$D'\iso A_5^r$. Every quotient of $A_5^r$ is again a direct product of
some of its simple factors \cite[Lemma~2.2]{BryantFoy1995}. Since $H'$
is the image of $D'$, the lemma is proved.
\end{proof}

Let
\begin{equation}\label{eq:J-limit}
 J_\infty=(L\circ L)\rtimes\langle t\rangle,
 \qquad
 L\circ L=(L\times L)/\langle(z,z)\rangle,
\end{equation}
where $t^2=1$ and $t$ induces $\theta$ on both copies of $L$.
This action is well defined since $\theta$ fixes $z$.

\begin{proposition}\label{prop:limit}
One has
\begin{equation}\label{eq:limit-formula}
 \mathfrak K_*=\form(J_\infty),\qquad
 \C_\infty=\form(J_\infty\times L)=\sform(\mathfrak K_*).
\end{equation}
Moreover, $|J_\infty|=14400$ and $|J_\infty\times L|=1\,728\,000$.
\end{proposition}

\begin{proof}
First we show that every infinite strictly descending chain in the
lattice \eqref{eq:known-lattice} has intersection $\mathfrak K_*$.
Indeed, such a chain contains no member of the finite chain at the
bottom of \eqref{eq:known-lattice}. Before it reaches some
$\mathfrak H_i$, it consists of formations $\mathfrak K_i$ with
strictly increasing indices; after that, it consists only of
formations $\mathfrak H_i$ with strictly increasing indices. If it
never reaches an $\mathfrak H_i$, then the indices of its members
$\mathfrak K_i$ are unbounded. In either case, \eqref{eq:known-lattice}
shows that $\mathfrak K_*$ is a lower bound of the chain, and that no
formation strictly above $\mathfrak K_*$ is a lower bound. Applying
this to the chain $\mathfrak E_k=\C_k\cap\mathfrak Q$, we get
\[
 \C_\infty\cap\mathfrak Q
 =\bigcap_{k\geq1}\mathfrak E_k=\mathfrak K_*.
\]
The formation $\C_\infty$ is $S_n$-closed and contains $\mathfrak R$.
Hence, applying Proposition~\ref{prop:recovery} to $\C_\infty$, we
obtain
\begin{equation}\label{eq:limit-recovery}
 \C_\infty=\mathfrak K_*\vee\mathfrak R.
\end{equation}

Now consider the subgroup
\[
 \{(x t^\varepsilon,y t^\varepsilon):x,y\in L,\
       \varepsilon\in\{0,1\}\}\sd A\times A.
\]
It is isomorphic to $(L\times L)\rtimes\langle t\rangle$, and its
quotient by
$\langle(z,z)\rangle$ is $J_\infty$. Hence $J_\infty\in\mathfrak Q$.
The two normal components of $J_\infty$ are isomorphic to $L$, they
have the same non-trivial central involution, and they have the same
outer-action character. Therefore their contributions to every $\mu_k$
cancel, and so $J_\infty\in\C_\infty\cap\mathfrak Q=\mathfrak K_*$.

The derived subgroup of $J_\infty$ is $L\circ L$, and it has a
non-trivial central involution. Therefore $J_\infty\notin\mathfrak P$
by Lemma~\ref{lem:S5-derived}. Since $\mathfrak P$ is the greatest
proper subformation of $\mathfrak K_*$ in \eqref{eq:known-lattice}, it
follows that $\form(J_\infty)=\mathfrak K_*$. Together with
\eqref{eq:limit-recovery}, this gives $\C_\infty=\form(J_\infty\times L)$.

Finally, each component of $J_\infty$ is isomorphic to $L$ and normal
in $J_\infty$, so
$\mathfrak R\subseteq\sform(\mathfrak K_*)$. Conversely, $\C_\infty$
is $S_n$-closed and contains $\mathfrak K_*$. These inclusions prove
the remaining equality in \eqref{eq:limit-formula}. The orders follow
from $|L\circ L|=120^2/2$ and $|\langle t\rangle|=2$. The proposition
is proved.
\end{proof}

\begin{remark}\label{rem:limit-scope}
Note that $J_\infty$ is an explicit generator of $\mathfrak K_*$, but
we do not claim that it is a generator of the smallest possible order.
We also do not determine the positions of the formations
$\mathfrak E_k$ in the lattice \eqref{eq:known-lattice}; they are not
needed for the computation of $\C_\infty$.
\end{remark}

\section{The hypothesis of \texorpdfstring{$S_n$}{Sn}-closedness in the
recovery formula}\label{app:boundary}

In this appendix we give an ordinary subformation of $\F$ for which
\eqref{eq:recovery} fails. For every $H\in\F$, we define the central
elements
\begin{equation}\label{eq:boundary-elements}
 \nu_0(H)=\prod_{E\in\Comp(H)}c_E,
 \qquad
 \gamma(H)=\prod_{\substack{E\in\Comp(H)\\\lambda_E=0}}c_E.
\end{equation}
Here $\lambda_E=0$ means that $\lambda_E$ is identically zero, and the
empty product is equal to $1$.

\Needspace{14\baselineskip}
\begin{lemma}\label{lem:boundary-natural}
For every epimorphism $f:H\epi J$ between groups in $\F$,
\[
 f(\nu_0(H))=\nu_0(J),\qquad f(\gamma(H))=\gamma(J).
\]
Consequently, each of
\[
 \mathfrak U_0=\{H\in\F:\nu_0(H)=1\},\qquad
 \mathfrak U_{\mathrm{in}}=\{H\in\F:\gamma(H)=1\}
\]
is an ordinary formation. Also, $\gamma(H)=1$ for $H\in\mathfrak Q$,
and $\gamma(H)=\nu_0(H)$ for $H\in\mathfrak R$.
\end{lemma}

\begin{proof}
We use the correspondence between the components of $H$ and $J$
established in the proof of Lemma~\ref{lem:epimorphism}. If a
component $E\iso L$ of $H$ is mapped onto a component isomorphic to
$L$, then $f|_E$ is an isomorphism and
\[
 \lambda_E=\lambda_{f(E)}\circ f.
\]
Since $f$ is surjective, $\lambda_E$ is zero if and only if
$\lambda_{f(E)}$ is zero. Moreover, $f(c_E)=c_{f(E)}$. If $E$ is
mapped onto $1$ or onto a group isomorphic to $A_5$, then $f(c_E)=1$.
Finally, components isomorphic to $A_5$ contribute $1$ to both
products. These observations prove both equalities. The classes
$\mathfrak U_0$ and $\mathfrak U_{\mathrm{in}}$ are non-empty and
closed under isomorphisms. They are closed under taking epimorphic
images by the equalities just proved, and they are closed under finite
subdirect products, since we may apply these equalities to the
coordinate projections, whose kernels have trivial intersection.

Now let $H\in\mathfrak Q$, and write $H=f(D)$, where $D\sd A^a$. If
$E\iso L$ is a component of $H$, then $E$ is the isomorphic image of
some diagonal $L$-block of $D'$. Choose a coordinate in this block.
Since $D$ is subdirect, some element of $D$ has $t$ as its coordinate
in this place; this element induces an outer automorphism of the
block, even when the block is twisted. Hence its image in $H$ gives a
non-zero value of $\lambda_E$. Thus no component of $H$ isomorphic to
$L$ contributes to $\gamma(H)$, and since the components isomorphic to
$A_5$ contribute $1$, we have $\gamma(H)=1$. The assertion for $\mathfrak R$ follows from the last
part of Lemma~\ref{lem:pure-L-closed}. The lemma is proved.
\end{proof}

\begin{proposition}\label{prop:recovery-boundary}
Let
\[
 B=(A\times L)/\langle(z,z)\rangle,
 \qquad \mathfrak B=\form(B).
\]
Then $|B|=14400$, $\mathfrak B\subseteq\F$, and
\begin{equation}\label{eq:recovery-fails}
 (\mathfrak B\cap\mathfrak Q)\vee(\mathfrak B\cap\mathfrak R)
 \subsetneq\mathfrak B.
\end{equation}
In particular, $\mathfrak B$ is not $S_n$-closed.
\end{proposition}

\begin{proof}
Since $B$ is a quotient of $A\times L$, we have $B\in\F$. The images
of the two copies of $L$ are distinct normal components of $B$
isomorphic to $L$, and their centres are identified with the same
non-identity element $z_*$. The outer-action character of the first
component is non-zero, and that of the second is zero. Therefore
\[
 \nu_0(B)=z_*^2=1,\qquad \gamma(B)=z_*\neq1.
\]
By Lemma~\ref{lem:boundary-natural}, $\mathfrak B\subseteq\mathfrak U_0$.
Every group in $\mathfrak B\cap\mathfrak Q$ satisfies $\gamma=1$, and
for the groups in $\mathfrak B\cap\mathfrak R$ the equality
$\gamma=\nu_0$ gives the same conclusion. Consequently, the join of
these two formations is contained in $\mathfrak U_{\mathrm{in}}$,
whereas $B\notin\mathfrak U_{\mathrm{in}}$. This proves
\eqref{eq:recovery-fails}.

Finally, $\nu_0(L)=z\neq1$, so $L\notin\mathfrak B$, although $B$ has
a normal component isomorphic to $L$. Hence $\mathfrak B$ is not
$S_n$-closed. Clearly, $|B|=240\cdot120/2=14400$. The proposition is
proved.
\end{proof}

\begin{remark}\label{rem:boundary-scope}
Note that the formation $\mathfrak B$ is one-generated. Thus
Proposition~\ref{prop:recovery-boundary} shows only that the
hypothesis of $S_n$-closedness in Proposition~\ref{prop:recovery}
cannot be omitted; it does not give an ordinary subformation of $\F$
which is not one-generated.
\end{remark}

\end{document}